\documentclass[12pt]{amsart}

\usepackage{fancyhdr,lipsum}
\usepackage{amssymb,amsmath,graphicx,color,textcomp, amsthm,bbm,bbold, enumerate,booktabs}
\usepackage{apptools}
\definecolor{Red}{cmyk}{0,1,1,0}

\definecolor{verde}{cmyk}{1,0,1,0}

\definecolor{loka}{cmyk}{.5,0,1,.5}
\definecolor{azul}{cmyk}{1,1,0,0}

\usepackage{hyperref}

\numberwithin{equation}{section}

\def\cal{\mathcal}

\newcommand{\eqd}{\stackrel{\tiny d}{=}}

\newcommand{\diam}{\mathrm{diam}}

\def\Ed{{\mathbb{E}}}
\def\Pd{{\mathbb{P}}}

\newcommand{\dist}{\mathrm{dist}}
\newcommand{\N}{\mathbb{N}}

\renewcommand{\P}{\mathbb{P}}

\renewcommand{\d}{\delta}
\newcommand{\e}{\varepsilon}

\newcommand{\s}{\sigma}

\newcommand{\be}{\begin{equation}}
\newcommand{\ee}{\end{equation}}

\newtheorem*{theorem*}{Theorem}
\newtheorem{theorem}{Theorem}

\newtheorem{definition}{Definition}
\newtheorem{lemma}{Lemma}
\newtheorem{corollary}{Corollary}

\title{Diameter of P.A. random graphs with edge-step functions}

\author{Caio Alves$^1$ }
\address{$^1$  Institute of Mathematics,
, University of Leipzig -- Augustusplatz 10, 04109 Leipzig
\newline
e-mail: {\itshape \texttt{caio.alves@math.uni-leipzig.de}}}

\author{Rodrigo Ribeiro$^{2}$}
\address{$^2$ PUC Chile, Av. Vicu\~{n}a Mackenna 4860, Macul, La Florida, Regi\'{o}n Metropolitana, Chile.
\newline
e-mail: {\itshape \texttt{rribeiro@impa.br}}}

\author{R{\'e}my Sanchis$^3$}
\address{$^3$Departamento de Matem{\'a}tica, Universidade Federal de Minas Gerais, Av. Ant\^onio
Carlos 6627 C.P. 702 CEP 30123-970 Belo Horizonte-MG, Brazil
\newline
e-mail: {\itshape \texttt{rsanchis@mat.ufmg.br}}}

\date{\today \\
    $^1$ Institute of Mathematics, University of Leipzig\\
   	$^2$ PUC Chile, Pontificia Universidad Cat\'{o}lica de Chile. \\
    $^3$ Departamento de Matem{\'a}tica, Universidade Federal de Minas Gerais.
}

\begin{document}

\begin{abstract} In this work we prove general bounds for the diameter of random graphs generated by a preferential attachment model whose parameter is a function $f:\N\to[0,1]$ that drives the asymptotic proportion between the numbers of vertices and edges. These results are sharp when $f$ is a \textit{regularly varying function at infinity} with strictly negative index of regular variation~$-\gamma$. For this particular class, we prove a characterization for the diameter that depends only on~$-\gamma$. More specifically, we prove that the diameter of such graphs is of order $1/\gamma$ with high probability, although its vertex set order goes to infinity polynomially. Sharp results for the diameter for a wide class of \textit{slowly varying functions} are also obtained.

\end{abstract}
\maketitle
\section{Introduction}

P. Erd\H{o}s and A. R\'{e}nyi in their seminal paper~\cite{ER59a} introduced the random graph model that now carries their name in order to solve combinatorial problems. However, the theory of Random Graphs as a whole has proven to be a useful tool for treating concrete problems as well. Any discrete set of entities whose elements interact in a pairwise fashion may be seen as a graph: the vertices represent the entities, and the edges, the possible interactions. This approach is nowadays intuitive and very fruitful. In the scenario where there exists some randomness on the interactions among the entities, random graphs became the natural tool to represent abstract or real phenomena.

From a mathematical/statistical point of view, the Erd\H{o}s-R\'{e}nyi model -- and many others related to it -- is \textit{homogeneous}, in the sense that its vertices are \textit{statistically indistinguishable}. However, the empirical findings of the seminal work of A. B\'{a}rabasi and R. \'{A}lbert \cite{BA99} suggested that many real-world networks are \textit{non-homogeneous}. They observed that such graphs were \textit{scale-free}, i.e., their degree sequence had a power-law distribution. The authors proposed a mechanism -- known as \textit{preferential attachment} -- that could explain the emergence of such highly skewed distributions. Roughly speaking, the idea is that some sort of popularity drives the interaction among the entities. 

Motivated mainly by these empirical findings, nowadays Preferential Attachment models (PA-models for short) constitute a well known class of random graph models investigated from both theoretical and applied perspectives. Recently, the preferential attachment mechanism has been generalized in many ways and combined with other rules of attachment, such as \textit{spatial proximity} \cite{JM15} and \textit{fitness of vertices} \cite{DO14}. It also arises naturally even in models where it is not entirely explicit such as the deletion-duplication models \cite{BM15, T15}, in which vertices' degree still evolve according to the PA-rule. Furthermore, the PA-models provide an interesting and natural environment for other random processes, such as \textit{bootstrap percolation, contact process and random walks}, see \cite{AF18, CD09, JM17} for recent examples of random processes whose random media is sampled from some PA-model.

When dealing with PA-models, there exists a set of natural questions that arises, such as the degree distribution and the order of the diameter. Their interest relies on modeling purposes and on the implications for the graph's combinatorial structures.

In this paper we address the latter topic on a PA-model which is a modification of the Barab\'asi-Albert model introduced in~\cite{bastat} (BA-model for short). The kind of result we pursuit is to show that some graph properties hold \textit{asymptotically almost surely} (a.a.s). Given a sequence of random graphs $\{G_t\}_{t \in \N}$, we say that a graph property $\mathcal{P}$ holds a.a.s, if
\[
\Pd \left( G_t \in \cal{P} \right) = 1-o(1)
\]
i.e., the probability of observing such property increases to~$1$ as~$t$ goes to infinity. For instance,~$\cal{P}$ may be the set of graphs having diameter less than the logarithm of the total number of vertices.

In order to offer a clearer discussion of our results, we introduce the model in the next subsection, then we discuss separately the properties which we want the graph to satisfy a.a.s, as well as the associated motivation.
\subsection{Preferential attachment model with an edge-step function}
The model we investigate here in its generality was proposed in \cite{ARSEdge17} and combines the traditional preferential attachment rule with a function called \textit{edge-step function} that drives the growth rate of the vertex set.

The model has one parameter $f$ which is a real non-negative function with domain given by~$\N$ and bounded by one on the $L_{\infty}$-norm, we will see~$f$ as a sequence of probabilities indexed by a time parameter~$t\in\N$. Without loss of generality and to simplify the expressions we deal with, we start the process from an initial graph $G_1$ consisting in one vertex and one loop. The model evolves inductively and at each step the next graph is obtained by performing one of the two stochastic operations defined below on the previous one:
\begin{itemize}
	\item \textit{Vertex-step} - Add a new vertex $v$ and add an edge $\{u,v\}$ by choosing $u\in G$ with probability proportional to its degree. More formally, conditionally on $G$, the probability of attaching $v$ to $u \in G$ is given by
	\[
	P\left( v \rightarrow u \middle | G\right) = \frac{\text{degree}(u)}{\sum_{w \in G}\text{degree}(w)}.
	\]
	\item \textit{Edge-step} - Add a new edge $\{u_1,u_2\}$ by independently choosing vertices $u_1,u_2\in G$ according to the same rule described in the vertex-step. We note that both loops and parallel edges are allowed.
	
\end{itemize}
The model alternates between the two types of operations according to a sequence $\{Z_t\}_{t\ge 1}$ of independent random variables such that $Z_t\eqd \mathrm{Ber}(f(t))$. We then define inductively a random graph process $\{G_t(f)\}_{t \ge 1}$  as follows: start with~$G_1$. Given $G_{t}(f)$, obtain $G_{t+1}(f)$ by either performing a \textit{vertex-step} on $G_t(f)$ when $Z_t=1$ or performing an \textit{edge-step} on $G_t(f)$ when $Z_t=0$. Notice that $f(t)$ is the probability of adding a new vertex at time $t$, thus the reader may think of $f$ as a \textit{vertex-step probability}.

Given $f$, its partial sum is an important quantity for us and we reserve the letter $F$ to denote it, i.e., $F$ is a function defined as
\begin{equation}
F(t) := 1+ \sum_{s=2}^tf(s).
\end{equation}
Observe that the edge-step function $f$ is intimately related to the growth of the vertex set. If we let $V_t$ denote the number of vertices added up to time $t$, then, if~$F(t)\xrightarrow{t\to\infty}\infty$, then a concentration of measures result will imply
\begin{equation}
V_t =1+ \sum_{s=2}^t Z_s \approx F(t),
\end{equation}
since~$\{Z_s\}_{s\ge 1}$ is a sequence of independent random variables. Thus, abusing the notation for a brief moment, we may write
\[
\frac{\mathrm{d}V_t}{\mathrm{d}t} = f(t).
\]
When the proper machinery has been settled, we will discuss in Section~$\ref{s:finalcomments}$ that some regularity should be imposed on $f$ in order to avoid some pathological behaviors. For now, we define a list of conditions we may impose on $f$ at different points of the paper in order to get the proper results. For instance, we say $f$ satisfies condition $(\mathrm{D})$ if it is \textit{non-increasing}. We define the further conditions:
\begin{equation}\tag{D$_0$}\label{def:conditionD0}
f \text{ is non-increasing and }\lim_{t\to \infty}f(t) =0;
\end{equation}
\begin{equation}\tag{S}\label{def:conditionS}
\sum_{s=1}^{\infty}\frac{f(s)}{s} < \infty;
\end{equation}
\begin{equation}\tag{L$_{\kappa}$}\label{def:conditionL}
\sum_{s=t^{1/13}}^{t}\frac{f(s)}{s} < (\log t)^{\kappa}, \text{ for all }t\in\N \text{ and some } \kappa \in (0,1);
\end{equation}
\begin{equation}\tag{RV$_{\gamma}$}\label{def:conditionRV}
\exists \gamma, \text{ such that }\forall a>0,\; \lim_{t \to \infty} \frac{f(at)}{f(t)} = \frac{1}{a^{\gamma}}.
\end{equation}
We must point out that for modeling purposes, conditions (D) and (\ref{def:conditionD0}) may be desirable. For instance, in the context of social networks, these conditions assure that the rate at which new individuals join the network is decreasing as the size of the network increases. Whereas, conditions (\ref{def:conditionS}) and (\ref{def:conditionL}) are related to the order of the maximum degree of $G_t(f)$. In \cite{ARSEdge17}, the authors point out that the maximum degree at time $t$ should be of order
\begin{equation}
t\cdot \exp \left \lbrace -\frac{1}{2}\sum_{s=2}^t\frac{f(s)}{s-1}\right \rbrace.
\end{equation}

A function satisfying condition (\ref{def:conditionRV}) is called \textit{regularly varying} at infinity and the exponent $\gamma$ is called the \textit{index of regular variation}. Functions in this class are well-studied in mathematics in many contexts and a variety of asymptotic results for them and their integrals is known due mainly to the theory developed by Karamata, see \cite{BGT89Book} for a complete reference. 

In general, we may say that this paper investigates how sensitive is the diameter to changes of $f$ and aims at a general characterization of such observable for a class of functions~$f$ that is as wide as possible. 

\subsection{Shaping the diameter} An important property of graphs which is also related to spread of rumors and connectivity of networks is the diameter, that is, the maximal graph distance between two vertices of said graph. Originally, investigating the diameter of real-world networks, the authors in \cite{SW98} observed that, although coming from different contexts, those networks usually have diameter of order less than the logarithm of the number of vertices, the so-called \textit{small-world} phenomena.

In this paper we address the issue of determining the order of the diameter of~$G_t(f)$. Our main goals in this subject are to obtain a characterization for the diameter imposing conditions on $f$ as weak as possible and also to obtain regimes for the diameter arbitrarily small but still preserving the scalefreeness of the graph.

In order to slow the growth of the diameter of PA-models, two observable play important roles: the maximum degree and the proportion of vertices with low degree. The former tends to concentrate connections on vertices with very high degree which acts in the way of shortening the diameter, since they attract connections to them; whereas the latter acts in the opposite way. In \cite{HHZ07} and \cite{FR07}, the authors have shown that in the configuration model with power-law distribution the diameter order is extremely sensitive to the proportion of vertices with degree 1 and 2.

One way to reduce the effect of low degree vertices on the diameter is via \textit{affine} preferential attachment rules, i.e., introducing a parameter $\d$ and choosing vertices with probability proportional to their degree plus $\d$. In symbols, conditionally on $G_t$, we connect a new vertex $v_{t+1}$ to an existing one $u$ with probability
\[
\Pd \left( v_{t+1} \rightarrow u \middle | G_t \right) = \frac{\text{degree}(u) + \d}{\sum_{w \in G_t}(\text{degree}(w) + \d)}.
\]
By taking a negative $\d$, the above rule increases the influence of high degree vertices and indeed decreases drastically the diameter's order. For instance, for positive $\d$ the diameter of $G_t$ is at least $\log(t)$, whereas for $\d <0$ the diameter of $G_t$ is at most $\log\log t$. See~\cite{DHH10} for several results on the diameter of different combinations for the affine preferential attachment rule.

Diminishing the effect of low degree vertices is not enough to break the growth of the diameter completely. The reason for that is, despite their low degree, these vertices exist in large amount. Even the existence of a vertex with degree close to $t$ at time $t$ may not be enough to freeze the diameter's growth. In \cite{MP14} the authors have proven that the maximum degree of a modification of the BA-model is of order $t$ at time $t$. However, the authors believe that this is not enough to obtain a diameter of order $\log \log t$, the reason being that this large hub still has to compete with a large number of low degree vertices.
\subsubsection{General bounds for the diameter} As said before, our goal is to develop bounds for the diameter of $G_t(f)$ with $f$ as general as possible. Under the condition of monoticity, we prove the following lower bound
\begin{theorem}[Lower bound on the diameter]\label{t:lowerbounddiam}  Let $f$ be an edge-step function satisfying condition $(\mathrm{D})$. Then
	\begin{equation}
	\label{e:diamlog}
	\Pd \left( \diam (G_t(f)) \ge \frac{1}{3} \left(\frac{\log t}{\log \log t} \wedge \frac{\log t}{-\log f(t)}\right)\right) = 1-o(1).
	\end{equation}
\end{theorem}
 Requiring more information on $f$, we prove upper bounds that, for a broad class of functions, are of the same order of the lower bounds given by the previous theorem. This is all summarized in the Theorem below.  
\begin{theorem}[Upper bound on the diameter]\label{t:upperbounddiam} Let $f$ be an edge-step function. Then
	\begin{enumerate}[(a)]
		\item if $f$ also satisfies conditions \eqref{def:conditionS} and \eqref{def:conditionD0} then
		\begin{equation*}
		\Pd \left(   \diam(G_t(f)) \le  2+ 6   \left(     \frac{\log t}{-\log\left(\sum_{s=t^{\frac{1}{13}}}^t \frac{f(s)}{s-1}  \right) }       \wedge    \frac{\log t}{\log\log t}                        \right)              \right) = 1- o(1);
		\end{equation*}
		\item  if $f$ satisfies condition \eqref{def:conditionL} then
		\begin{equation*}
		\Pd \left(   \diam(G_t(f)) \le  2+ \frac{6 }{1-\kappa}    \frac{\log t}{\log\log t} \right) = 1- o(1).
		\end{equation*}
	\end{enumerate}
\end{theorem}
\subsubsection{The class regularly of varying functions}In \cite{ARSEdge17}, the authors prove a characterization of the empirical degree distribution of graphs generated by $f$ satisfying condition~(\ref{def:conditionRV}), for~$\gamma \in [0,1)$. More specifically, they prove that the degree distribution of such graphs obeys a power law distribution whose exponent depends only on the index of regular variation $-\gamma$.

A byproduct of our general bounds is a similar characterization for the diameter. For edge-step functions satisfying conditions \eqref{def:conditionD0} and \eqref{def:conditionRV} for $\gamma \in (0,\infty)$ the graphs generated by such functions have constant diameter and its order depends only on the index of regular variation $-\gamma$. We state this result in the theorem below
\begin{theorem}[Diameter of regularly varying functions]\label{t:rv} Let $f$ be an edge-step function satisfying conditions \eqref{def:conditionD0} and \eqref{def:conditionRV}, for~$\gamma \in (0,\infty)$. Then,
	\[
	\Pd\left( \frac{1}{4\gamma} \le \diam(G_t(f)) \le \frac{100}{\gamma}+2 \right) = 1-o(1).
	\]
\end{theorem}
\subsubsection{The class of slowly varying functions} The case when $\gamma = 0$ is richer in terms of possible orders of the diameter and does not admit a nice characterization as the one we obtain for positive~$\gamma$. In this settings, we present another consequence of our bounds for particular subclasses of the class of slowly varying functions. Let us first define the subclass of functions and later state how our results fit these specific classes.
\begin{equation}
L := \left \lbrace  f \text{ is an edge-step function such that }f(t) = \frac{1}{\log^\alpha(t)}, \text{ for some }\alpha > 0\right \rbrace;
\end{equation}
\begin{equation}
E := \left \lbrace f \text{ is an edge-step function such that } f(t) = e^{-\log^{\alpha }(t)}, \text{ for some }\alpha \in (0,1)\right\rbrace.
\end{equation}
It is straightforward to verify that functions belonging to the set above defined are slowly varying. For functions belonging to the two subclasses $L$ and $E$, our results have the following consequences, verifiable through elementary calculus,
\begin{corollary}\label{cor:sv1}Let $f$ be an edge-step function.
	\begin{enumerate}[(a)]
		\item if $f$ belongs to $L$, with $\alpha \le 1$, then
		\begin{equation*}
		\label{eq:logpeq}   
		\Pd \left(\frac{1}{3}     \frac{\log t}{\log \log t} \leq \diam(G_t(f))   \leq \frac{8}{\alpha}     \frac{\log t}{\log \log t}  \right) = 1- o(1);
		\end{equation*}
		\item if $f$ belongs to $L$, with $\alpha > 1$, then
		\begin{equation*}
		\label{eq:exloggran}
		\Pd \left(\frac{1}{3\alpha}\frac{\log t}{\log \log t}  \leq  \diam(G_t(f))   \leq \frac{7}{\alpha-1} \frac{\log t}{\log \log t} \right) = 1-o(1);
		\end{equation*}
		\item if $f$ belongs to $E$, then
		\begin{equation*}
		\label{eq:expeq}   
		\Pd \left(C_{\alpha}^{-1}(\log t)^{1-\alpha} \leq \diam(G_t(f))   \leq C_{\alpha}(\log t)^{1-\alpha}  \right) = 1- o(1),
		\end{equation*}
		for some $C_{\alpha} \ge 1$.
	\end{enumerate}
\end{corollary}

\subsection{Main technical ideas} In order to prove the existence of some given subgraph in the (affine) BA-random graphs a key ingredient is usually to use the fact that two given vertices~$v_i$ and~$v_j$ -- born at times~$i,j\in\N$ respectively -- may be connected only at one specific time-step, since (assuming~$i<j$) the model's dynamic only allows $v_j$ to connect to $v_i$ at the moment in which~$v_j$ is created. This property facilitates the computation of the probability of the occurrence of a given subgraph and decreases the combinatorial complexity of the arguments. In \cite{bollrior, eggemann} the authors estimate the number of triangles and cherries (paths of length~$3$) of the (affine) BA-model and their argument relies heavily on this feature of the model. In our case, however, the \textit{edge-step} prevents an application of such arguments, since a specific subgraph may appear at any time after the vertices have been added. 

Another difficulty in our setup is the degree of generality we work with. Our case replaces the parameter $p \in (0,1]$ in the models investigated in \cite{ARS17, CLBook, CF03} by any non-negative real function~$f$ with $\| f \|_{\infty} \le 1$. The introduction of such function naturally increases the complexity of any analytical argument one may expect to rely on, as will be clear during our estimates for the vertices' degree, and makes it harder to discover \textit{threshold} phenomena. This is the reason why in our work the Karamata's Theory of \textit{regularly varying functions} is crucial in order to prove sharper results.

In order to prove Theorem~\ref{t:lowerbounddiam}, we apply the second moment method on the number of \textit{isolated paths}. This approach demands correlation estimations for the existence of two such paths in~$G_t(f)$, which we do under the assumption of $f$ being monotonic only.

For the proof of Theorem \ref{t:upperbounddiam}, a key step is to obtain a lower bound for the maximum degree, since high degree vertices tend to attract vertices to themselves. However, as said above, the degree of generality of $f$ makes all analytical arguments more involved. Then, in order to overcome part of the heavy computation we have to deal with in our setup, in Lemma \ref{l:thestar} we obtain lower bounds for the maximum degree by constructing a monotonic coupling with the case in which $f \equiv p \in (0,1)$. This coupling allow us to transpose known results about the maximum degree of $G_t(p)$ to the general case of $G_t(f)$.

We then combine Lemma \ref{l:thestar} with a lower bound for the degree of earlier vertices, which is obtained by estimation of negative moments of a given vertex's degree, in order to show that, under conditions~\eqref{def:conditionS} or~\eqref{def:conditionL}, long paths of younger vertices are unlikely and older vertices are all very close in graph distance. Finally, using results from the Karamata's theory of regularly varying functions, we verify that this broad class of functions satisfies our assumptions, proving Theorem~\ref{t:rv}.

\subsection{Organization} 
In Section~\ref{s:tecest}, we prove technical estimates for the degree of a given vertex, which are needed for the upper bound on the diameter. Section~\ref{s:lower} is devoted to the general lower bound for the diameter, i.e., for the proof of Theorem~\ref{t:lowerbounddiam}. We prove the upper bound for the diameter, Theorem~\ref{t:upperbounddiam}, in Section~\ref{s:upper}. Finally, in Section~\ref{s:rv} we show what our results say for the class of regularly varying functions. We end the paper at Section~\ref{s:finalcomments} with some comments on the affine version of our model and a brief discussion on what may happen to the model if some regularity conditions are dropped.

\subsection{Notation}\label{ss:notation}
We let~$V(G_t(f))$ and~$E(G_t(f))$ denote the set of \textit{vertices} and \textit{edges} of~$G_t(f)$, respectively. Given a vertex~$v\in V(G_t(f))$, we will denote by~$D_t(v)$ its \textit{degree} in~$G_t(f)$. We will also denote by~$\Delta D_t(v)$ the \emph{increment} of the discrete function~$D_t(v)$ between times~$t$ and~$t+1$, that is,
\begin{equation*}
\Delta D_t(v) =D_{t+1}(v) -D_t(v).
\end{equation*}
When necessary in the context, we may use~$D_G(v)$ to denote the degree of~$v$ in the graph~$G$.

Given two sets~$A,B\subseteq V(G_t(f))$, we let~$\{A\leftrightarrow B\}$ denote the event where there exists an edge connecting a vertex from~$A$ to a vertex from~$B$. We denote the complement of this event by~$\{A\nleftrightarrow B\}$. We let~$\dist(A,B)$ denote the graph distance between~$A$ and~$B$, i.e. the minimum number of edges that a path that connects~$A$ to~$B$ must have. When one of these subsets consists of a single vertex, i.e. $A=\{v\}$, we drop the brackets from the definition and use~$\{v\leftrightarrow B\}$ and~$\dist(v,B)$, respectively.

Regarding constants, we let $C,C_1,C_2,\dots$ be positive real numbers that do not depend on~$t$ whose values may vary in different parts of the paper. The dependence on other parameters will be highlighted throughout the text.

Since our model is inductive, we use the notation $\mathcal{F}_t$ to denote the $\s$-algebra generated by all the random choices made up to time $t$. In this way we obtain the natural filtration $\mathcal{F}_0 \subset \mathcal{F}_1 \subset \dots $ associated to the process.

\section{Technical estimates for the degree}\label{s:tecest}

In this section we develop technical estimates related to the degree of a given vertex. We begin by stating one of the most fundamental identities in the study of preferential attachment models: the conditional distribution of the increment of the degree of a given vertex. Given~$v\in V(G_t(f))$, we have
\begin{equation}
\label{eq:degvcond}
\begin{split}
\P\left(  \Delta D_t(v)  =0 \middle |\mathcal{F}_t   \right) &= f(t+1)\left( 1-\frac{ D_t(v)}{2t}    \right) +(1-f(t+1))\left( 1-\frac{ D_t(v)}{2t}    \right)^2
\\
\P\left(  \Delta D_t(v) =1 \middle |\mathcal{F}_t   \right) &= f(t+1)\frac{ D_t(v)}{2t}+2(1-f(t+1))\frac{ D_t(v)}{2t}\left( 1-\frac{ D_t(v)}{2t}    \right),
\\
\P\left(  \Delta D_t(v)=2 \middle |\mathcal{F}_t   \right) &= (1-f(t+1))\frac{ D_t(v)^2}{4t^2}.
\end{split}
\end{equation}
To see why the above identities hold true,  observe for example that in order for $\Delta D_t(v)=0$, either a vertex step was taken, and the vertex did not connect to~$v$, or an edge step was taken and neither of the endpoints of the new edge connected to~$v$. The other equations follow from analogous reasonings. As a direct consequence, we obtain
\begin{equation}
\label{eq:degvconde}
\begin{split}
\Ed \left[\Delta D_t(v)\middle| \mathcal{F}_t\right] 
& =  1\cdot f(t+1)\cdot\frac{D_t(v)}{2t} + 1\cdot 2(1-f(t+1))\frac{D_t(v)}{2t}\left( 1- \frac{D_t(v)}{2t} \right)  \\&\quad+ 2\cdot (1-f(t+1))\frac{D^2_t(v)}{4t^2}  \\
& = \left( 1-\frac{f(t+1)}{2} \right)\frac{D_t(v)}{t}.
\end{split}
\end{equation}
Using the above equation repeatedly one obtains, conditioned on the event where the vertex~$v$ is born at time~$t_0$,
\begin{equation}
\label{eq:degexpec}
\begin{split}
\Ed \left[D_t(v)\right] &= \Ed \left[\Ed \left[ D_t(v)\middle| \mathcal{F}_{t-1}\right] \right]
\\
&= \left( 1+\frac{1}{t-1}-\frac{f(t)}{2(t-1)} \right)\Ed \left[ D_{t-1}(v) \right]
\\
&=\prod_{s=t_0 }^{t-1}\left( 1+\frac{1}{s}-\frac{f(s+1)}{2s} \right).
\end{split}
\end{equation}

The next lemma gives a lower bound for the maximum degree of $G_t(f)$ when $f$ decreases to zero. This result will be crucial in the proof of Theorem \ref{t:upperbounddiam}, which gives general upper bounds for the diameter. Its proof involves the construction of a coupling between $\{G_t(f)\}_{t\geq 1}$ and~$\{G_t(p)\}_{t\geq 1}$, where~$G_t(p)$ is constructed from the constant function equal to~$p\in(0,1]$ for every~$t\in\N$. This coupling will be monotonic in the sense that, if $G_t(p)$ has a vertex with high degree, then so has ~$G_t(f)$.

\begin{lemma}\label{l:thestar} Let $f$ be an edge-step function such that $f(t) \searrow 0$ as $t$ goes to infinity. Then, for any fixed $\e>0$,
	\[
	\Pd \left( D_{\mathrm{max}}\left( G_t(f) \right) < t^{1-\e}\right) \xrightarrow{t\to\infty} 0.
	\]
\end{lemma}

Let us describe the general idea behind the proof of the above result. Exclusively in the proof of this Lemma we will denote as~$i\in\N$ the~$i$-th vertex to be added to the graph. We couple the degree sequence of the first $m$ vertices added by the process $\{G_s(f)\}_{s \ge 1}$, for some large constant $m$, to the degree sequence of vertices $\{ t^\varepsilon+1, \dots, t^\varepsilon+m\}$ in $\{G_s(p)\}_{s \ge 1}$ for large $t$ and small $\varepsilon$. The coupling is constructed in such way that the $i$-th vertex of $\{G_s(f)\}_{s \ge 1}$ always has larger degree than the $(t^\varepsilon+i)$-th vertex in $\{G_s(p)\}_{s \ge 1}$. This is possible for two reasons, the first being due to the fact that, since vertices are chosen with probability proportional to their degrees and both graphs have the same number of edges, if vertex $i$ in $G_t(f)$ has larger degree than vertex $t^\varepsilon +i$ in $G_t(p)$, then $i$ is more likely to increase its degree in the next step. Moreover, when using this fact to show stochastic domination by induction, the first step follows because, if at the moment in which process $\{G_s(p)\}_{s \ge 1}$ adds its $(t^\varepsilon +i)$-th vertex process $\{G_s(f)\}_{s \ge 1}$ has already added vertex $i$, then~$i$ has degree at least one. The second reason is the fact that $f$ decreases to zero. This property guarantees that $f(t)$ is always smaller than $p$ after some time, which allows us to couple both process in a way that process $\{G_s(p)\}_{s \ge 1}$ takes less edge-steps than $\{G_s(f)\}_{s \ge 1}$, and therefore vertices in $\{G_s(p)\}_{s \ge 1}$ have less chances to increase their degrees. We finally combine the coupling with previous results about $G_t(p)$ in~\cite{ARS17}, which imply that, \textit{w.h.p}, some vertex $j \in G_t(p)$ added at some time between $t^\varepsilon+1$ and $t^\varepsilon +m$ has degree at least $t^{(1-\varepsilon)(1-p/2)}$. 

Although the coupling above described might be believable, its formalization may be a little involved since we have to deal with times in which the two graph processes perform different graph operations (vertex or edge-step), as well as the case when $\{G_s(f)\}_{s \ge 1}$ has less than $m$ vertices at the time when $\{G_s(p)\}_{s \ge 1}$ adds its $t^\varepsilon$-th vertex.

\begin{proof} The case in which $f$ decreases fast enough to zero so that $\sum_{s=1}^{\infty}f(s) < \infty$ is simpler. By the hypothesis on $f$ and the fact that in this case $|V_t|$ is the sum of independent Bernoulli random variables, it follows that $\{\exp\{|V_t|\}\}_{t\geq 1}$ is limited in $L_1$. Thus, by Markov's inequality
	\[
	\Pd\left( |V_t| > 2\log t \right) \le \frac{\sup_t \Ed \exp \{|V_t|\}}{t^2}.
	\]
	And since the sum of all degrees at time $t$ is $2t$, we have w.h.p that $G_t(f)$ has at most $2\log t$ vertices, and by the pigeonhole principle it follows in this case that there exists a vertex of degree at least~$t/\log t$.
	
	We now consider the case in which the average number of vertices goes to infinity. We begin by applying Theorem~$2$ of \cite{ARS17}, choosing $m, R$ large enough, $p$ small enough, and letting $j = t^{\varepsilon}$ in order to obtain, for large enough $t$, the bound 
	\begin{equation}\label{eq:teorem2}
	\Pd\left( \sum_{i=1}^{m} D_{G_t(p)} (j+i) < t^{1-\varepsilon}\right) \le t^{-C},
	\end{equation}
	for some $C$ depending on $\varepsilon, m, p$ and $R$.
	That is, with probability at least~$1-t^{-C}$, for small~$p$ and large $m$, at least one of the vertices $\{j+1, \dots, j+m \}$ has degree at least $t^{1-\varepsilon}/m$ in~$G_t(p)$.

	 Let $t_0 := \min \{ t \ge 0 \; ; \;  f(t)\le p \}$, which is well-defined since $f$ decreases to zero, and consider~$t$ in the definition of $j$ be larger than $t_0$.
	 We will couple $\{G_t(f)\}_{t \in \N}$ and $\{G_t(p)\}_{t \in \N}$ for all $i \in \{1, \dots, m\}$ in such way that
	 \begin{equation}\label{ineq:degs}
	 D_{G_{s}(p)}(j+i) \le D_{G_{s}(f)}(i), \forall s\ge 0, \text{w.h.p.}
	 \end{equation}
	 In order to formalize the coupling, we assume that $|V(G_{t^\varepsilon}(f))| \ge m$ and that our probability space is large enough so we have at our disposal two independent sequences of r.v.s $\{(U_s^{(1)}, U_s^{(2)})\}_{s \in \N}$ and $\{U'_s\}_{s \in \N}$ such that $U^{(1)}\eqd U_s^{(2)} \eqd U' \eqd \mathrm{Uni}[0,1]$, and that all these variables are mutually independent. We will use the sequence $\{U_s'\}_{s\in \N}$ to control simultaneously the sequence of edge and vertex-steps we take in each graph process. The sequence $\{(U_s^{(1)}, U_s^{(2)})\}_{s \in \N}$ on the other hand will be used to select vertices in both graphs.
	 
	 We proceed by induction. Define the degree of a yet-unborn vertex as~$0$. Then in the event where~$|V(G_{t^\varepsilon}(f))| \ge m$, \eqref{ineq:degs} holds up to time~$s=t^\e$. Suppose we have succeeded in coupling $\{G_r(f)\}_{r \le s}$ and $\{G_r(p)\}_{r \le s}$ in such way that \eqref{ineq:degs} is satisfied for~$s\geq t^\e$. Now we have in the same probability space both graphs $G_s(f)$ and $G_s(p)$ with the property that \eqref{ineq:degs} holds. To advance the induction, we need to know how to generate $G_{s+1}(f)$ and $G_{s+1}(p)$. To do so, let us first introduce some notation. For $i \in \{1, \dots, m\}$, define inductively the random intervals
	 \begin{equation*}
	 I_{s}^{(1)} := \left[0, \frac{D_{G_s(f)}(1)}{2s}\right] =: [b_0(s), b_1(s)], \; 		I_{s}^{(i)}\equiv[b_{i-1}(s),b_{i}(s)] := \left[b_{i-1}(s), b_{i-1}(s)+ \frac{D_{G_s(f)}(i)}{2s}\right],
	 \end{equation*}
	 and for each $I_s^{(i)}$ we let $I_{s,p}^{(j+i)}$ be
	 \begin{equation}
	 I_{s,p}^{(j+i)} := \left[b_{i-1}(s), b_{i-1}(s) + \frac{D_{G_s(p)}(j+i)}{2s} \right].
	 \end{equation}
	 Notice that, since \eqref{ineq:degs} holds, $ I_{s,p}^{(j+i)} \subset  I_{s}^{(i)}$, and that, since the total degree of~$G_t(f)$ is always~$2s$, all the above intervals are contained in~$[0,1]$.
	 
	 Now we need to know which kind of step we take on each graph and we do this according to $U'_{s+1}$ (recall that this variable is independent of the whole past).\\
	 
	 \noindent \underline{ $U'_{s+1} \le f(s+1)$.} In this case we perform a vertex-step in \textit{both graphs}. We use~$U^{(1)}_{s+1}$ in order to select two vertices (one in each graph) to connect the new vertex that was born in each graph.  We increase the degree of vertex $i \in G_s(f)$ (resp. $j+i \in G_s(p)$) if $U^{(1)}_{s+1} \in I_{s}^{(i)}$ (resp.  $U^{(1)}_{s+1} \in I_{s,p}^{(j+i)}$ ). If~$U^{(1)}_{s+1}$ does not belong to any interval $I_{s}^{(i)}$ for $i=1,\dots,m$ (resp. any interval $ I_{s,p}^{(j+i)}$ ), we let~$G_{s+1}(f)$ (resp.~$G_{s+1}(p)$) be constructed in some arbitrary manner that preserves its law. By the construction of the intervals it is clear that \eqref{ineq:degs} is preserved in time $s+1$.

	  \noindent \underline{ $f(s+1) \le U'_{s+1} \le p$.} In this case, we perform an edge-step on $G_{s+1}(f)$ and a vertex-step on $G_{s+1}(p)$. As before, we sample $U^{(1)}_{s+1}$ and if $U^{(1)}_{s+1} \in I_{s,p}^{(j+i)}$ we connect the new vertex of $G_{s+1}(p)$ to $j+i$. As for $G_{s+1}(f)$, we interpret $U^{(1)}_{s+1}$ as indicating the endvertex of one of the halves of the new edge.  Then we sample $U^{(2)}_{s+1}$ and perform a similar procedure as before, but only considering the intervals with respect to $f$, in order to choose the second half of the new edge added to $G_{s+1}(f)$. We again construct~$G_{s+1}(f)$ and~$G_{s+1}(p)$ in some arbitrary manner in the case the respective uniform variables do not fall in any of the associated intervals. 
	   
	   \noindent \underline{ $p \le U'_{s+1}$.} In this case we perform an edge-step in both graphs. To do so, it is enough to repeat of procedure of selecting vertices using both $U_{s+1}^{(1)}$ and $U_{s+1}^{(2)}$.
	   
	   Finally, with Markov's inequality we bound the probability that $|V(G_{t^\varepsilon}(f))| \le m$. Recalling that~$|V(G_{t^\varepsilon}(f))| $ is the sum of independent Bernoulli random variables, we obtain
		\begin{equation}
		\Pd\left( |V(G_{t^\varepsilon}(f))| \le m\right) \le e^m \Ed \left[ \exp\{-|V(G_{t^\varepsilon}(f))|\}\right] \le \exp\{m-(1-e^{-1})F(t^\varepsilon)\},
		\end{equation}
		which completes the proof since~$\e$ was chosen arbitrarily.
\end{proof}
Our objective now is to obtain a polynomial lower bound for the degree of older vertices, which will be important in the proof of the upper bound for the diameter in Theorem~\ref{t:upperbounddiam}. We begin with an upper bound for the expectation of the multiplicative inverse of the degree. Recall the definition of the process~$(Z_t)_{t\geq 1}$, consisting of independent Bernoulli variables that dictate whether a vertex-step or an edge-step is performed at time~$t$.
\begin{lemma}\label{l:invdege} Given any edge-step function~$f$, consider the process $\{G_t(f)\}_{t\ge 1}$. Denote by~$v_i$ the vertex born at time~$i\in\N$. We have
	\begin{equation}
	\label{eq:invdege}
	\Ed\left[    (D_t(v_i))^{-1}Z_i\right] \le f(i) \left(  \frac{t-1}{i}       \right)^{-\frac{1}{6}}.
	\end{equation}
	And consequently
	\begin{equation}
	\label{eq:invdegmarkov}
	\Pd\left( D_t(v_i) \leq \left(\frac{t-1}{i}\right)^{\frac{1}{12}}  \middle | Z_i=1\right)\leq     \left(  \frac{t-1}{i} \right)^{-\frac{1}{12}}.
	\end{equation}
\end{lemma}
\begin{proof} 
	If~$Z_i=1$, then for every~$s \geq i$ we have that
	\[
	\Delta D_s(v_i)\geq 0,\quad D_{s+1}(v_i)\leq D_s(v_i)+2 \leq 3D_s(v_i),
	\]
	and that~$D_s(v_i)$ is~$\mathcal{F}_s$ measurable. Together with~\eqref{eq:degvconde}, these facts imply, on the event where~$Z_i=1$, 
	\begin{equation}
	\begin{split}
	\Ed\left[       \frac{1}{D_{s+1}(v_i)}     - \frac{1}{D_{s}(v_i)}      \middle |       \mathcal{F}_s\right]
	&=
	\Ed\left[      - \frac{\Delta D_s(v_i)}{D_{s+1}(v_i)D_{s}(v_i)}         \middle |       \mathcal{F}_s\right]
	\\
	&\leq
	-\frac{1}{3(D_{s}(v_i))^2}\left(1-\frac{f(s+1)}{2}\right)\frac{D_{s}(v_i)}{s}
	\\
	&\leq 
	-\frac{1}{6s}\frac{1}{D_{s}(v_i)},
	\end{split}
	\end{equation}
	since~$f(k)\leq 1$ for very~$k\in\N$. Therefore,
	\begin{equation}
	\Ed\left[    Z_i \cdot ( D_{s+1}(v_i)  )^{-1}        \middle |       \mathcal{F}_s\right] \leq Z_i(1-(6s)^{-1})(D_{s}(v_i))^{-1}.
	\end{equation}
	Iterating the above argument from~$i$ until~$t$, we obtain
	\begin{equation}
	\label{eq:deginv1}
	\Ed\left[    Z_i \cdot ( D_{t}(v_i)  )^{-1}       \right] \leq f(i) \prod_{s=i}^{t-1}\left(1-\frac{1}{6s}\right) \leq f(i) \exp\left\{  -\frac{1}{6} \sum_{s=i}^{t-1}  \frac{1}{s}     \right\}\leq f(i)\left(  \frac{t-1}{i}       \right)^{-\frac{1}{6}},
	\end{equation}
	proving~\eqref{eq:invdege}. We then obtain~\eqref{eq:invdegmarkov} by an elementary application of the Markov inequality:
	\begin{equation}
	\label{eq:deginv2}
	\begin{split}
	\Pd\left( D_t(v_i) \leq \left(\frac{t-1}{i}\right)^{\frac{1}{12}}  \middle | Z_i=1\right)&= 
	\Pd\left(  ( D_{t}(v_i)  )^{-1} \geq \left(\frac{t-1}{i}\right)^{-\frac{1}{12}}   \middle | Z_i=1   \right)
	\\&\leq 
	\left(\frac{t-1}{i}\right)^{\frac{1}{12}} \Ed\left[ ( D_{t}(v_i)  )^{-1} \middle | Z_i=1 \right]
	\\&\leq    \left(  \frac{t-1}{i} \right)^{-\frac{1}{12}}.
	\end{split}
	\end{equation}
\end{proof}

We now provide an elementary consequence of the above result, which uses the union bound in order to show that, with high probability, every vertex born before time~$t^{\frac{1}{12}}$ has degree at least~$t^{\frac{1}{15}}$ by time~$t$.
\begin{lemma}
	\label{l:deglowerbound}
	Using the same notation as in Lemma~\ref{l:invdege} we have, for every edge-step function~$f$ and for sufficiently large~$t\in\N$,
	\begin{equation}
	\label{eq:deglowerbound}
	\Pd\left( \exists i \in \N, 1\leq i \leq t^{\frac{1}{12}},\text{ such that }Z_i=1\text{ and }D_t(v_i)\leq t^{\frac{1}{15}} \right)\leq    Ct^{-\frac{1}{144}}.
	\end{equation}
	\begin{proof}
		By the union bound and equation~\eqref{eq:invdegmarkov}, we have that the probability in the left hand side of~\eqref{eq:deglowerbound} is smaller than or equal to
		\begin{equation}
		\label{eq:deglowerbound2}
		\begin{split}
		\sum_{i=1}^{t^{\frac{1}{12}}} \Pd\left( Z_i=1,D_t(v_i)\leq t^{\frac{1}{15}} \right)
		&\leq
		\sum_{i=1}^{t^{\frac{1}{12}}} f(i) \Pd\left( D_t(v_i) \leq \left(\frac{t-1}{i}\right)^{\frac{1}{12}}  \middle | Z_i=1\right)
		\\
		&\leq \sum_{i=1}^{t^{\frac{1}{12}}} \left(\frac{t-1}{i}\right)^{-\frac{1}{12}} 
		\\
		&\leq C t^{-\frac{1}{12}+\frac{1}{12}\left(1-\frac{1}{12}\right)}
		\\
		&\leq C t^{-\frac{1}{144}},
		\end{split}
		\end{equation}
		finishing the proof of the Lemma.
	\end{proof}
\end{lemma}

\section{General lower bound for the diameter: proof of Theorem~\ref{t:lowerbounddiam}}\label{s:lower}

The proof of Theorem~$\ref{t:lowerbounddiam}$ follows a second moment argument. The idea is to count the number of ``long'' (the specific size depending on~$f$ and~$t$) \emph{isolated paths} in~$G_t(f)$. We begin by stating two key lemmas for the proof. The former,  Lemma~\ref{l:numsnakes}, states that the expected number of isolated paths goes to infinity with~$t$. The latter, Lemma \ref{l:isopathcor}, guarantees that the presence of a specific isolated path is almost independent from the presence of some other given isolated path whenever said paths are disjoint. This ``almost independence'' makes the second moment of the number of such paths very close to the first moment squared, which allows us to apply Paley-Zygmund inequality.

We start by defining precisely what we mean by an isolated path.
\begin{definition}[$t$-isolated path] Let~$l$ be a positive integer. Let $\vec{t} = (t_1,..,t_l)$ be a vector of distinct positive integers. We say that this vector corresponds to an \emph{isolated path} $\{v_{t_1},\dots,v_{t_l}\}$ in~$G_t(f)$ if and only if:
	\begin{itemize}
		\item $t_l\le t$;
		
		\item $t_i<t_j$ whenever $1\le i < j \le l$;
		
		\item during each time~$t_i$, $i=1,\dots,l$, a vertex-step is performed;
		
		\item for every integer $k\le l$, the subgraph induced by the vertices $\{v_{t_i}\}_{1\le i \le k}$ is connected in~$G_{t_k}(f)$;
		
		\item for $i=1,\dots,l-1$, the degree of~$v_{t_i}$ in~$G_t(f)$ is~$2$. The degree of~$v_{t_l}$ in~$G_t(f)$ is~$1$.
	\end{itemize}
\end{definition}
In other words, an $t$-isolated path $\{v_{t_i}\}_{1\le i \le l}$ is a path where each vertex~$v_{t_i}$, for $i=2,\dots,l$, is born at time~$t_i$ and makes its first connection to its predecessor~$v_{t_{i-1}}$. Other than that, no other vertex or edge gets attached to~$\{v_{t_i}\}_{1\le i \le l}$. We will denote~$\{v_{t_i}\}_{1\le i \le l}$ by~$v_{\vec{t}}$.

Given~$\xi \in(0,1)$, we will denote by~$\mathcal{S}_{l,\xi}(t)$ the set of all $t$-isolated paths in~$G_t(f)$ of size~$l$ whose vertices were created between times~$\xi t$ and~$t$. Having all the necessary notation, we now state the two technical lemmas needed for the proof of Theorem \ref{t:lowerbounddiam}.

\begin{lemma}[Average number of $t$-isolated paths]\label{l:numsnakes}Let $f$ be a non-increasing edge-step function, then, for any $0<\xi<1$ and any integer~$l$, the following lower bound holds:
	\begin{equation}
	\label{eq:nctbound1}
	\Ed \left[|\mathcal{S}_{l,\xi}(t)|\right] \ge \binom{(1-\xi)  t}{l}\frac{f(t)^l}{(2t)^{l-1}}\left(1-\frac{2l}{ \xi t} \right)^t.
	\end{equation}
	Furthermore, for
	\begin{equation}
	\label{eq:nctboundl}
	l \leq \frac{1}{3} \left(\frac{\log t}{\log \log t} \wedge \frac{\log t}{-\log f(t)}\right),
	\end{equation}
	we have that, for sufficiently large~$t$,
	\begin{equation}
	\label{eq:nctbound2}
	\Ed \left[|\mathcal{S}_{l,\xi}(t)|\right] \ge t^{\frac{1}{4}}.
	\end{equation}
\end{lemma}

\begin{lemma}[Correlation between $t$-isolated paths]\label{l:isopathcor} Let~$l$ be such that~\eqref{eq:nctboundl} is satisfied. For two $t$-isolated paths with disjoint time vectors~$\vec{t}$ and~$\vec{r}$, we have
	\begin{equation}
	\label{eq:pvtvrcomp}
	\frac{\Pd \left(v_{\vec{t}} \in \mathcal{S}_{l,\xi}(t) \right)\Pd \left(v_{\vec{r}} \in \mathcal{S}_{l,\xi}(t) \right)}{ \Pd \left(v_{\vec{t}}, v_{\vec{r}}\in \mathcal{S}_{l,\xi}(t) \right)}  = 1 + o(1).
	\end{equation}
\end{lemma}
We postpone the proof of these two lemmas, since their proofs are technical. Below, we show how the main result of this section follows from them.
\begin{proof}[Proof of Theorem \ref{t:lowerbounddiam}]  Recall that Paley-Zygmund's inequality (see e.g. section~$5.5$ of \cite{lyonyuval}) assures us that, for any $ 0\le \theta \le 1$,
	\begin{equation}\label{eq:paley}
	\mathbb{P} \left( |\mathcal{S}_{l,\xi}(t)| > \theta \Ed \left[|\mathcal{S}_{l,\xi}(t)|\right] \right) \ge (1-\theta)^2\frac{\left(\Ed \left[|\mathcal{S}_{l,\xi}(t)|\right]\right)^2}{\Ed \left[|\mathcal{S}_{l,\xi}(t)|^2\right]}.
	\end{equation}
	
	On the other hand, let~$l$ be such that~\eqref{eq:nctboundl} is satisfied, and consider~$\xi \in(0,1)$. Since it is impossible for two non disjoint  and non equal isolated paths to exist at the same time, we have that
	\begin{equation}
	\begin{split}
	\Ed \left[|\mathcal{S}_{l,\xi}(t)|^2 \right] & = \Ed \left[ \left( \sum_{\vec{t}} \mathbb{1} \left\lbrace v_{\vec{t}} \in \mathcal{S}_{l,\xi}(t) \right\rbrace \right)\left( \sum_{\vec{r}} \mathbb{1} \left\lbrace v_{\vec{r}} \in \mathcal{S}_{l,\xi}(t) \right\rbrace \right) \right] \\ 
	& = \sum_{\substack{ \vec{t},\vec{r} \\ \text{disjoint}}} \Pd \left( v_{\vec{t}},v_{\vec{r}} \in \mathcal{S}_{l,\xi}(t) \right) + \Ed \left[|\mathcal{S}_{l,\xi}(t)| \right],
	\end{split}
	\end{equation}
	and that
	\[
	\left( \Ed \left[ |\mathcal{S}_{l,\xi}(t)| \right] \right)^2 = \sum_{\substack{ \vec{t},\vec{r} \\ \text{disjoint}}} \Pd \left( v_{\vec{t}} \in \mathcal{S}_{l,\xi}(t) \right)\Pd \left( v_{\vec{r}} \in \mathcal{S}_{l,\xi}(t)\right) + \sum_{\substack{ \vec{t},\vec{r} \\ \vec{r} \cap \vec{t} \neq \emptyset}} \Pd \left( v_{\vec{t}} \in \mathcal{S}_{l,\xi}(t) \right)\Pd \left( v_{\vec{r}} \in \mathcal{S}_{l,\xi}(t) \right).
	\]
	Therefore, by lemmas \ref{l:numsnakes} and~\ref{l:isopathcor}, we obtain
	\begin{equation}
	\label{e:snakebixao}
	\begin{split}
	\frac{\Ed \left[|\mathcal{S}_{l,\xi}(t)|^2 \right]}{\left( \Ed \left[ |\mathcal{S}_{l,\xi}(t)| \right] \right)^2} 
	&=
	\frac{\sum_{\substack{ \vec{t},\vec{r} \\ \text{disjoint}}} \Pd \left( v_{\vec{t}},v_{\vec{r}} \in \mathcal{S}_l(t) \right) }{\left( \Ed \left[ |\mathcal{S}_{l,\xi}(t)| \right] \right)^2}
	+
	\frac{ \Ed \left[|\mathcal{S}_{l,\xi}(t)| \right]}{\left( \Ed \left[ |\mathcal{S}_{l,\xi}(t)| \right] \right)^2}
	\\
	&\le \frac{\sum_{\substack{ \vec{t},\vec{r} \\ \text{disjoint}}} \Pd \left( v_{\vec{t}},v_{\vec{r}} \in \mathcal{S}_{l,\xi}(t) \right) }{\sum_{\substack{ \vec{t},\vec{r} \\ \text{disjoint}}} \Pd \left( v_{\vec{t}} \in \mathcal{S}_{l,\xi}(t) \right)\Pd \left( v_{\vec{r}} \in \mathcal{S}_{l,\xi}(t) \right)} + \frac{1}{\Ed \left[ |\mathcal{S}_{l,\xi}(t)| \right]} 
	\\
	&\le \frac{\sum_{\substack{ \vec{t},\vec{r} \\ \text{disjoint}}} (1+o(1))\Pd \left( v_{\vec{t}} \in \mathcal{S}_{l,\xi}(t) \right)\Pd \left( v_{\vec{r}} \in \mathcal{S}_{l,\xi}(t) \right) }{\sum_{\substack{ \vec{t},\vec{r} \\ \text{disjoint}}} \Pd \left( v_{\vec{t}} \in \mathcal{S}_{l,\xi}(t) \right)\Pd \left( v_{\vec{r}} \in \mathcal{S}_{l,\xi}(t) \right)} + \frac{1}{\Ed \left[ |\mathcal{S}_{l,\xi}(t)| \right]} 
	\\
	&\le
	1+ o(1),
	\end{split}
	\end{equation}
	then by choosing $\theta = \theta(t)=\Ed \left[|\mathcal{S}_{l,\xi}(t)|\right]^{-1/2}$ and combining \eqref{e:snakebixao} and \eqref{eq:paley} we conclude the proof.
\end{proof}
We now prove the previously stated lemmas.
\begin{proof}[Proof of Lemma \ref{l:numsnakes}]
	The random variable $|\mathcal{S}_{l,\xi}(t)|$ can be written as
	\begin{equation}
	\label{e:numsnakes}
	|\mathcal{S}_{l,\xi}(t)| = \sum_{t_1 <t_2<\dots<t_l } \mathbb{1}\left\lbrace v_{\vec{t}} \in \mathcal{S}_{l,\xi}(t)\right\rbrace \quad \implies \quad \Ed \left[|\mathcal{S}_{l,\xi}(t)|\right] = \sum_{t_1 <t_2<\dots<t_l } \mathbb{P}\left( v_{\vec{t}} \in \mathcal{S}_{l,\xi}(t)\right).
	\end{equation}
	So it will be important to obtain a proper lower bound for $\mathbb{P}\left( v_{\vec{t}} \in \mathcal{S}_{l,\xi}(t)\right)$. Given a time vector of an isolated path~$\vec{t}=(t_1,\dots,t_l)$ such that $ \xi t\le t_1 <t_2<\dots<t_l\le t$, it follows that
	\begin{equation}
	\label{eq:pvectlower}
	\mathbb{P}\left( v_{\vec{t}} \in \mathcal{S}_{l,\xi}(t)\right) \ge \frac{f(t)^l}{(2t)^{l-1}}\left(1-\frac{2l}{  \xi t } \right)^t,
	\end{equation}
	since in order for~$v_{\vec{t}}$ to be in~$\mathcal{S}_{l,\xi}(t)$, we need to assure that~$l$ vertices are born exactly at times~$t_1,\dots,t_l$ (which happens with probability greater than~$f(t)^l$, by the monotonicity of~$f$), that~$v_{t_i}$ connects to~$v_{t_{i-1}}$ for every~$i=2,\dots,l$ (which happens with probability greater than~$(2t)^{-(l-1)}$), and that no other vertex or edge connects to~$v_{\vec{t}}$ until time~$t$ (which happens with probability greater than~$\left(1-\frac{2l}{  \xi t } \right)^t$).
	
	Finally, by counting the number of possible ways to choose $ t_1 <t_2<\dots<t_l$ so that~$t_i \in [\xi t,t]$ for $1\le i \le l$, we obtain~\ref{eq:nctbound1}. 
	
	We now assume~$l$ to be such that~(\ref{eq:nctboundl}) holds. Stirling's formula gives us
	\begin{equation}
	\nonumber
	\begin{split}
	\log \left(\binom{ (1-\xi )t}{l}\right) &\ge C+(1-\xi )t\log((1-\xi )t) -(1-\xi )t +\frac{\log((1-\xi )t)}{2}
	\\
	&\quad-l\log l + l - \frac{\log l }{2}
	\\
	&\quad- ((1-\xi )t-l)\log((1-\xi )t-l)+(1-\xi )t-l    - \frac{\log((1-\xi )t-l)}{2},
	\end{split}
	\end{equation}
	Since~$l \ll t$, we obtain
	\begin{equation}
	\label{eq:binlowerbound}
	\log \left(\binom{ (1-\xi )t}{l}\right) \geq l\log t - l \log l -Cl.
	\end{equation}
	Equation~\eqref{eq:pvectlower} then implies, again for sufficiently large~$t$,
	\begin{equation}
	\label{eq:logpvectlower}
	\log\left(\mathbb{P}\left( v_{\vec{t}} \in\mathcal{S}_{l,\xi}(t)\right) \right)
	\ge  l \log f(t) - (l-1)\log t -Cl.
	\end{equation}
	Combining the above inequality with~\eqref{eq:nctbound1}, \eqref{eq:nctboundl}, and~\eqref{eq:binlowerbound} gives us that, for large enough~$t$,
	\begin{equation*}
	\begin{split}
	\Ed \left[|\mathcal{S}_{l,\xi}(t)|\right]   \ge  \exp\left\lbrace  l\log t - l \log l + l \log f(t) - (l-1)\log t -Cl\right\rbrace
	\geq
	t^{\frac{1}{4}},
	\end{split}
	\end{equation*}
	since
	\[
	l \log l \leq \frac{1}{3} \frac{\log t}{\log \log t}    \log\log t \left(1 -\frac{\log\log\log t  +\log 3}{\log \log t}\right)     \leq \frac{\log t}{3}  ,
	\]
	which finishes the proof of the lemma.
\end{proof}

\begin{proof}[Proof of Lemma \ref{l:isopathcor}] We begin the proof by making some remarks regarding the existence of a specific $t$-isolated path and by introducing a new notation. Given a $t$-isolated path $v_{\vec{t}}=\{v_{t_j}\}_{1\le j \le l}$, we denote by~$D_r(v_{\vec{t}})$ the sum of the degrees of each of its vertices at time~$r$, i.e.\ :
	\begin{equation}
	D_r(v_{\vec{t}}) = \sum_{t_i \in \vec{t}} D_r(v_{t_i}),
	\end{equation}
	where we assumed that $D_r(v_{t_i})=0$ if $t_i>r$. Note that if~$t_l\le r\le t$ and~$v_{\vec{t}}$ has size $l$, then $D_r(v_{\vec{t}}) = 2l-1$. Furthermore, by the same reasoning as in~\eqref{eq:degvcond},
	\begin{equation}
	\label{eq:deltavect}
	\begin{split}
	\Pd \left( \Delta  D_s(v_{\vec{t}}) =0 ~\middle|\mathcal{F}_s\right) &=
	f(s+1)\left( 1-\frac{ D_s(v_{\vec{t}})}{2s}    \right) +(1-f(s+1))\left( 1-\frac{ D_s(v_{\vec{t}})}{2s}    \right)^2
	\\
	&= 1 - \left(1 - \frac{f(s+1)}{2}    -(1-f(s+1))\frac{D_s(v_{\vec{t}})}{4s}\right)\frac{D_s(v_{\vec{t}})}{s}.
	\end{split}
	\end{equation}
	
	Regarding $t$-isolated paths, observe that, in order for a $t$-isolated path $v_{\vec{t}}$ to appear the following must happen:
	\begin{itemize}
		\item a vertex~$v_{t_1}$ must be be created at time~$t_1$, which happens with probability~$f(t_1)$;
		\item between times~$t_1+1$ and~$t_2 -1$ there can be no new connection to~$v_{t_1}$, which, by~\eqref{eq:deltavect}, happens with probability
		\[
		\prod_{r_1= t_1 +1}^{t_2 -1}\left(1- \left(1 - \frac{f(r_1)}{2}-(1-f(r_1))\frac{1}{4(r_1-1)}\right)\frac{1}{r_1-1}\right);
		\]
		\item In general, at time~$t_k$ a vertex~$v_{t_k}$ is created and makes its first connection to~$v_{t_{k-1}}$,  no new connection is then made to~$\{v_{t_j}\}_{1\le j \le k}$ between times~$t_k + 1$ and~$t_{k+1} -1$ for every~$k=2,\dots,l-1$, all this happens with probability equal to
		\begin{align*}
		\lefteqn{f(t_k)\frac{1}{2(t_k-1)}\prod_{r_k= t_k +1}^{t_{k+1} -1}\left(1- \left(1 - \frac{f(r_k)}{2}-(1-f(r_k))\frac{D_{r_k}(v_{\vec{t}}) }{4(r_k-1)}\right)\frac{D_{r_k}(v_{\vec{t}}) }{r_k-1}\right)}\quad
		\\&=
		f(t_k)\frac{1}{2(t_k-1)}\prod_{r_k= t_k +1}^{t_{k+1} -1}\left(1- \left(1 - \frac{f(r_k)}{2}-(1-f(r_k))\frac{2k-1 }{4(r_k-1)}\right)\frac{2k-1}{r_k-1}\right);
		\end{align*}
		\item finally, a vertex~$v_{t_l}$ is born at time~$t_l$, connects to~$v_{t_{l-1}}$ and no new connection is made to~$\{v_{t_j}\}_{1\le j \le l}$ between times~$t_l+1$ and~$t$.
	\end{itemize}
	This implies
	\begin{equation}
	\label{eq:oneskaneprob}
	\begin{split}
	\lefteqn{\Pd \left(v_{\vec{t}} \in \mathcal{S}_{l,\xi}(t) \right)} \\& =f(t_1)
	\prod_{r_1= t_1 +1}^{t_2 -1}\left(1- \left(1 - \frac{f(r_1)}{2}-(1-f(r_1))\frac{1}{4(r_1-1)}\right)\frac{1}{r_1-1}\right)
	\\
	&\quad\times\dots\times
	f(t_k)\frac{1}{2(t_k-1)}\prod_{r_k= t_k +1}^{t_{k+1} -1}\left(1- \left(1 - \frac{f(r_k)}{2}-(1-f(r_k))\frac{2k-1 }{4(r_k-1)}\right)\frac{2k-1}{r_k-1}\right)
	\\
	&\quad\times
	\dots\times f(t_l)\frac{1}{2(t_l-1)}\prod_{r_l= t_l +1}^{t}\left(1- \left(1 - \frac{f(r_l)}{2}-(1-f(r_l))\frac{2l-1 }{4(r_l-1)}\right)\frac{2l-1}{r_l-1}\right).
	\end{split}
	\end{equation}
	Moreover, given two time vectors~$\vec{r}$ and~$\vec{t}$, we note that $\Pd \left(v_{\vec{t}}, v_{\vec{r}} \in \mathcal{S}_{l,\xi}(t) \right)$ is only nonzero if~$\vec{r}$ and~$\vec{t}$ have either disjoint or identical sets of entries.
	
	Then, to prove the lemma we will make a comparison between the two probabilities terms $\Pd \left(v_{\vec{t}}, v_{\vec{r}} \in \mathcal{S}_{l,\xi}(t) \right)$ and~$\Pd \left(v_{\vec{t}} \in \mathcal{S}_{l,\xi}(t) \right)\Pd \left(v_{\vec{r}} \in \mathcal{S}_{l,\xi}(t) \right)$. We can write both these terms as products in the manner of~\eqref{eq:oneskaneprob}. We can then compare the terms from these products associated to each time $s\in[\xi t,t]$. There are two cases we must study.
	
	\textit{Case 1: $s \in \vec{t}$ but $s \notin \vec{r}$ ($s \notin \vec{t}$ but $s \in \vec{r}$.).}
	
	The product term related to time~$s$ in~$\Pd \left(v_{\vec{t}}, v_{\vec{r}} \in \mathcal{S}_{l,\xi}(t) \right)$ is 
	\begin{equation}
	\label{e:compsnake1}
	\frac{f(s)}{2(s-1)},
	\end{equation}
	since a new vertex is created and then makes its first connection specifically to the latest vertex of~$\vec{t}$. On the other hand, the  term related to time~$s$ in $\Pd \left(v_{\vec{t}} \in \mathcal{S}_{l,\xi}(t) \right)\Pd \left(v_{\vec{r}} \in \mathcal{S}_{l,\xi}(t) \right)$ is 
	\begin{equation*}
	\frac{f(s)}{2(s-1)}
	\left(1- \left(1-\frac{f(s)}{2}-(1-f(s))\frac{D_{s-1}(v_{\vec{r}})}{4(s-1)}\right)\frac{D_{s-1}(v_{\vec{r}})}{s-1}\right),
	\end{equation*}
	since the term  related to~$s$ in the product form of $\Pd \left(v_{\vec{t}} \in \mathcal{S}_{l,\xi}(t) \right)$ continues to be equal to~\eqref{e:compsnake1}, but the related term in  $\Pd \left(v_{\vec{r}} \in \mathcal{S}_{l,\xi}(t) \right)$ is
	\begin{equation}
	\label{e:compsnake2}
	\left(1- \left(1-\frac{f(s)}{2}-(1-f(s))\frac{D_{s-1}(v_{\vec{r}})}{4(s-1)}\right)\frac{D_{s-1}(v_{\vec{r}})}{s-1}\right).
	\end{equation}
	The above expression is the term that will appear regarding the time~$s$ in the fraction in the left hand side of~\eqref{eq:pvtvrcomp}.
	This case occurs $2l$ times since the isolated paths are disjoint. Thus, recalling that $s\in [\xi  t,t]$, $l \leq 3^{-1}\log(t)/\log(\log(t))$, and that the degree of each isolated path is at most~$2l-1$, we obtain that there exist constants $C_1,C_2>0$ such that we can bound the product of all the terms of the form~\eqref{e:compsnake2} from above by
	\[
	\left(1 - \frac{C_1}{t} \right)^{2l},
	\]
	and from below by
	\[
	\left(1 - \frac{C_2 l}{t} \right)^{2l}.
	\]
	Observe that both products go to $1$ as $t$ goes to infinity.
	
	\textit{Case 2: $s \notin \vec{t}$ and $s \notin \vec{r}$.}
	
	In $\Pd \left(v_{\vec{r}} \in \mathcal{S}_{l,\xi}(t) \right)$ as well as in $\Pd \left(v_{\vec{t}} \in \mathcal{S}_{l,\xi}(t) \right)$ we see terms of the form~\eqref{e:compsnake2}, since we must avoid the isolated paths in both events. But in the term related to $\Pd \left(v_{\vec{t}}, v_{\vec{r}} \in \mathcal{S}_{l,\xi}(t) \right)$ we actually observe
	\[
	\left(1- \left(1-\frac{f(s)}{2}-(1-f(s))\frac{(D_{s-1}(v_{\vec{t}})+D_{s-1}(v_{\vec{r}}))}{4(s-1)}\right)\frac{(D_{s-1}(v_{\vec{t}})+D_{s-1}(v_{\vec{r}}))}{s-1}\right)
	,
	\]
	since we must guarantee that neither isolated path receives a connection. We note however that
	\begin{align*}
	\lefteqn{\left(1- \left(1-\frac{f(s)}{2}-(1-f(s))\frac{D_{s-1}(v_{\vec{r}})}{4(s-1)}\right)\frac{D_{s-1}(v_{\vec{r}})}{s-1}\right)}
	\\
	&\phantom{**********************}\times\left(1- \left(1-\frac{f(s)}{2}-(1-f(s))\frac{D_{s-1}(v_{\vec{t}})}{4(s-1)}\right)\frac{D_{s-1}(v_{\vec{t}})}{s-1}\right)\quad
	\\
	&=
	\left(1- \left(1-\frac{f(s)}{2}-(1-f(s))\frac{(D_{s-1}(v_{\vec{t}})+D_{s-1}(v_{\vec{r}}))}{4(s-1)}\right)\frac{(D_{s-1}(v_{\vec{t}})+D_{s-1}(v_{\vec{r}}))}{s-1}\right)
	\\
	&\phantom{**********************}\times\left(1+O\left(\frac{l^2}{t^2}\right)\right),
	\end{align*}
	since~$D_{s-1}(v_{\vec{t}}),D_{s-1}(v_{\vec{r}})\leq 2l-1$ and~$s \geq \xi t$. In the fraction in the left hand side of~\eqref{eq:pvtvrcomp}, we will then have $\Theta(t)$ terms of the form  
	\begin{equation*}
	\left(1+O\left(\frac{l^2}{t^2}\right)\right),
	\end{equation*}
	But again, as in Case~$1$, their product goes to 1 as~$t\to\infty$ since $l^2 = o(t)$. This finishes the proof of the lemma.
\end{proof}

\section{General upper bound for the diameter: proof of Theorem~\ref{t:upperbounddiam}}\label{s:upper}

In this section we provide a proof for Theorem \ref{t:upperbounddiam}. The main idea is to use Lemmas~\ref{l:thestar} and~\ref{l:deglowerbound} to show that, with high probability, all vertices born up to time~$t^{\frac{1}{12}}$ are in a connected component with diameter~$2$. We then use a first moment estimate (Lemma~\ref{l:pathfirstmoment2} below) to show that the lengths of the paths formed by newer vertices have the desired upper bound. As in the previous section, we will state only the essential lemma needed for the proof of the main result and postpone its proof to the end of this section.

Given~$k\in\N$ and $k$ time steps~$s_1,\dots,s_k\in\N$ such that~$s_1<\dots< s_k$, we say that the vector of time steps~$\vec{s}=(s_1,\dots,s_k)$ is a \emph{vertex path} if in the process~$\{G_t(f)\}_{t\geq 1}$, at each time~$s_j$, for~$j=2,\dots,k$, a vertex is born and makes its first (vertex-step) connection to the vertex born at time~$s_{j-1}$. We denote by
\[
\{s_1\leftarrow s_2\leftarrow\dots\leftarrow s_k        \}
\]
the event where~$\vec{s}$ is a vertex path.  Moreover, given~$k,t_0,t\in\N$, denote by~$\mathcal{V}_{k,t_0}(t)$ the set of all vertex-paths of length~$k$ whose vertices were born between times~$t_0$ and~$t$. 

The lemma below gives an upper bound to the average cardinality of ~$\mathcal{V}_{k,t_0}(t)$ in terms of $k, t_0$ and $t$.

\begin{lemma}\label{l:pathfirstmoment2}
	Using the notation above defined, we have, for~$k,t_0,t\in\N$,
	\begin{equation}\label{eq:pathfirstmoment5}
	\Ed\left[ |\mathcal{V}_{k,t_0}(t)|\right] \leq C_1\exp\left\{  2\log t  -(k-2)\left(\log(k-2)  +C_2 +\log\left(\sum_{j=t_0}^{t}\frac{f(j)}{j-1} \right) \right) \right\}.
	\end{equation}
\end{lemma}

Next we show how the above lemma and estimates on vertices' degree implies Theorem~\ref{t:upperbounddiam}.
\begin{proof}[Proof of Theorem~\ref{t:upperbounddiam}] We prove part (a) first.
	
	\noindent \underline{Proof of part (a):} Recall that, in this part of the theorem, $f$ is under condition (\ref{def:conditionS}), which holds if $\sum_{s=2}^{\infty}f(s)/s$ is finite.
	
	 For~$t_0,t\in\N$ and~$\delta\in(0,1)$, let~$A_\delta(t_0,t)$ be the event where every vertex born before time~$t_0$ has degree at least~$t^{\delta}$ in~$G_t(f)$. For~$\varepsilon\in(0,1)$, let~$B_\varepsilon(t)$ be the event where there exists a vertex~$v$ in~$V(G_t(f))$ such that~$D_t(v)\geq t^{1-\e}$. Then, by Lemmas~\ref{l:thestar} and~\ref{l:deglowerbound}, we have
	\begin{equation}
	\label{eq:etdt}
	\Pd\left(  A_{\frac{1}{15}}(t^{\frac{1}{12}},t) \right) = 1-o(1),\quad \Pd\left(  B_{\frac{1}{30}}(t) \right)= 1-o(1).
	\end{equation}
	We have, on the event where two vertices~$u_1,u_2\in V(G_t(f))$ are such that~$D_t(u_1)\geq t^{15^{-1}}$ and~$D_t(u_2)\geq t^{1-30^{-1}}$,
	\begin{equation}\label{ineq:uvconnected}
	\begin{split}
	\Pd\left(  u_1 \nleftrightarrow u_2 \text{ in  }G_{2t}(f)     \right) & \le
	\prod_{s=t+1}^{2t}\left(  1-   \frac{(1-f(s))t^{15^{-1}}t^{1-30^{-1}}}{2(s-1)^{2}}     \right)
	\\
	&\leq\exp\left\{            -\frac{t^{15^{-1}}t^{1-30^{-1}}}{16t^2}     \cdot t      \right\}
	\\
	&=
	\exp\left\{            -\frac{t^{-30^{-1}}}{16}         \right\}.
	\end{split}
	\end{equation}
	
	Recall the notation~$v_i$ symbolizing the vertex born at time~$i\in\N$. Together with~\eqref{eq:etdt} and the union bound, the above equation implies
	\begin{equation}
	\nonumber
	\begin{split}
	\lefteqn{\Pd\left(  \exists i,j\in\N, \text{ such that }1\leq i < j \leq t^{12^{-1}},Z_i=Z_j=1, \text{ and } \dist(v_i,v_j)>2 \text{ in }G_{2t}(f)   \right)}\phantom{**********}
	\\
	&\leq
	o(1)+\sum_{1\leq i<j\leq t^{\frac{1}{12}}}\Pd\left( 
	A_{\frac{1}{15}}(t^{\frac{1}{12}},t), B_{\frac{1}{30}}(t)
	, \dist(v_i,v_j)>2 \text{ in }G_{2t}(f)   \right)
	\\
	&\leq
	o(1)+t^{-6^{-1}}\exp\left\{            -\frac{t^{-30^{-1}}}{16}         \right\}
	\\
	&= o(1).
	\end{split}
	\end{equation}
	This implies the existence of a constant $C_1>0$ such that the probability of there existing two vertices born before time~$C_1  t^{12^{-1}}$ such that the distance between said vertices is larger than~$2$ in~$G_t(f)$ is polynomially small in~$t$. We now turn our attention to vertices born after~$t^{\frac{1}{13}}$. We will use Lemma~\ref{l:pathfirstmoment2} in order to bound the probability of there existing long vertex-paths formed by vertices born after~$t^{\frac{1}{13}}$, the notion of a ``long'' path being~$f$-dependent. Let~$V_t(t^{\frac{1}{13}})$ denote the set of all vertices of~$V(G_t(f))$ born before time~$t^{\frac{1}{13}}$, let~${\bf d}_{\mathrm{max}}(t^{\frac{1}{13}},t)$ be the length of a maximal vertex-path of vertices born between times~$t^{\frac{1}{13}}$ and~$t$. Let~$u_1,u_2\in V(G_t(f))$. Since~$G_t(f)$ is connected,
	\begin{equation}
	\label{eq:diameter}
	\begin{split}
	\dist(u_1,u_2)&\leq  \dist(u_1,V_t(t^{\frac{1}{13}}))+ \diam(V_t(t^{\frac{1}{13}}))+ \dist(u_2,V_t(t^{\frac{1}{13}}))
	\\
	&\leq 2 {\bf d}_{\mathrm{max}}(t^{\frac{1}{13}},t) +  \diam(V_t(t^{\frac{1}{13}})).
	\end{split}
	\end{equation}
	But we know that~$ \diam(V_t(t^{\frac{1}{13}}))\leq 2$ with high probability. Bounding~${\bf d}_{\mathrm{max}}(t^{\frac{1}{13}},t) $ then gives us an a.a.s. upper bound for the diameter of~$G_t(f)$.
	
	Now, given~$t\in\N$, if
	\begin{equation}
	k\geq 3   \left(     \frac{\log t}{-\log\left(\sum_{s=t^{\frac{1}{13}}}^t \frac{f(s)}{s-1}  \right) }       \wedge    \frac{\log t}{\log\log t}                        \right),
	\end{equation}
	then, by Lemma~\ref{l:pathfirstmoment2}, and since~$\log   \left(   \sum_{j=t^{\frac{1}{13}}}^{t}\frac{f(j)}{j-1} \right)      $ is eventually negative for large~$t$ (recall again that~$f$ satisfies (\ref{def:conditionS})), we have,
	\begin{equation}
	\label{eq:upperdiamsum}
	\begin{split}
	\lefteqn{\Pd\left(      {\bf d}_{\mathrm{max}}(t^{\frac{1}{13}},t) > k                 \right)}
	\phantom{****}
	\\
	&\leq
	\Ed\left[ |\mathcal{V}_{k,t^{1/13}}(t)| \right] 
	\\
	&\leq 
	C_1\exp\left\{  2\log t  -(k-2)\left(\log(k-2)  +C_2  -\log\left(       \sum_{j=t^{\frac{1}{13}}}^{t}\frac{f(j)}{j-1} \right)        \right)                              \right\}
	\\
	&\leq C t^{-\frac{1}{2}}.
	\end{split}
	\end{equation}
	The above upper bound and (\ref{eq:diameter}) proves part (a).
		
	\noindent \underline{Proof of part (b):} Recall that in this part, we are under condition (\ref{def:conditionL}), which holds if, for~$\kappa\in(0,1)$ and every~$t\in\N$ sufficiently large, one has
	\begin{equation}
	\label{eq:diamcond2}
	\sum_{s=t^{1/13}}^{t}\frac{f(s)}{s}  < (\log t)^{\kappa}.
	\end{equation}
	Then, let $k$ be so that
	\[
	k \geq \frac{3 }{1-\kappa}    \frac{\log t}{\log\log t}                 .     
	\]
	We then obtain, again by Lemma~\ref{l:pathfirstmoment2}, for sufficiently large~$t$,
	\begin{equation}
	\label{eq:upperdiamlogk}
	\begin{split}
	\lefteqn{\Pd\left(      {\bf d}_{\mathrm{max}}(t^{\frac{1}{13}},t) > k                 \right)}\quad\quad
	\\
	&\leq
	\Ed\left[ |\mathcal{V}_{k,t^{1/13}}(t)| \right] 
	\\
	&\leq 
	C_1\exp\left\{  2\log t  -(k-2)\left(\log(k-2)  +C_2  -\log\left(       \sum_{j=t^{\frac{1}{13}}}^{t}\frac{f(j)}{j-1} \right)        \right)                              \right\}
	\\
	&\leq
	C_1\exp\left\{  2\log t  -\frac{3}{1-\kappa}\frac{\log t}{\log\log t}        \left(    (1-\kappa)\log \log t    - C - \log\log\log t  \right)                              \right\}
	\\
	&\leq C t^{-\frac{1}{2}}.
	\end{split}
	\end{equation}
	Finally, the above upper bound together with~\eqref{eq:diameter} finishes the proof of the Theorem.
\end{proof}
Before proving Lemma \ref{l:pathfirstmoment2}, we will need an intermediate result.
\begin{lemma}
	\label{l:pathfirstmoment}
	Using the notation above, we have, for each vector~$\vec{s}=(s_1,\dots,s_k)$ and step function~$f$,
	\begin{eqnarray}
	\label{eq:pathfirstmoment}
	\Pd \left( s_1\leftarrow s_2\leftarrow\dots\leftarrow s_k   \right) \leq
	f(s_1)\frac{s_k-1}{s_1+1}\prod_{m=2}^k\frac{f(s_m)}{2(s_{m}-1)}.
	\end{eqnarray}
\end{lemma}
\begin{proof}
	Consider the events~$\{         s_1\leftarrow s_2\leftarrow\dots\leftarrow s_{k-1}    \}$ and~$\{    s_{k-1}\leftarrow  s_k   \}$, defined analogously as the event in the above equation, but for the vectors~$(s_1,\dots,s_{k-1})$ and~$(s_{k-1},s_k)$ respectively. We have
	\begin{equation}
	\label{eq:pathfirstmoment2}
	\begin{split}
	\Pd \left( s_1\leftarrow s_2\leftarrow\dots\leftarrow s_k   \right) 
	&=
	\Ed \left[ \mathbb{1}\{s_1\leftarrow s_2\leftarrow\dots\leftarrow s_{k-1}\}   \Pd\left( s_{k-1} \leftarrow s_k   \middle | \mathcal{F}_{s_k  -1}\right)\right]
	\\
	&=
	\Ed \left[ \mathbb{1}\{s_1\leftarrow s_2\leftarrow\dots\leftarrow s_{k-1}\}   f(s_k) \cdot \frac{D_{s_k-1}(v_{s_{k-1}})}{2(s_k-1)}\right].
	\end{split}
	\end{equation}
	But, crucially, conditioned on the event where a vertex is born at time~$s_{k-1}$, the degree of said vertex at time~$s_k-1$ depends only on the connections made after time~$s_{k-1}$, and is therefore independent of the indicator function above. We then obtain, by~\eqref{eq:degexpec},
	\begin{equation}
	\label{eq:pathfirstmoment3}
	\begin{split}
	\lefteqn{\Pd \left( s_1\leftarrow s_2\leftarrow\dots\leftarrow s_k   \right) }\quad
	\\
	&=
	\Ed \left[ \mathbb{1}\{s_1\leftarrow s_2\leftarrow\dots\leftarrow s_{k-1}\}   f(s_k) \cdot \frac{D_{s_k-1}(v_{s_{k-1}})}{2(s_k-1)}\middle | Z_{s_{k-1}}=1\right]f(s_{k-1})
	\\
	&=f(s_{k-1})f(s_k)\Pd \left(s_1\leftarrow s_2\leftarrow\dots\leftarrow s_{k-1}   \middle | Z_{s_{k-1}}=1\right) 
	\Ed \left[ \frac{D_{s_k-1}(v_{s_{k-1}})}{2(s_k-1)}\middle | Z_{s_{k-1}}=1\right]
	\\
	&=
	\Pd \left(s_1\leftarrow s_2\leftarrow\dots\leftarrow s_{k-1} \right)\frac{f(s_k)}{2(s_k-1)}\prod_{m=s_{k-1} }^{s_k-2}\left( 1+\frac{1}{m}-\frac{f(m+1)}{2m} \right)
	\\
	&\leq
	\Pd \left(s_1\leftarrow s_2\leftarrow\dots\leftarrow s_{k-1} \right)\frac{f(s_k)}{2(s_k-1)}\exp\left\{   \sum_{m=s_{k-1}}^{s_k-2} \left(\frac{1}{m}-\frac{f(m+1)}{2m}\right)   \right\},
	\end{split}
	\end{equation}
	by elementary properties of the exponential function. Iterating the above argument and recalling that the vertex~$s_1$ is born with probability~$f(s_1)$, we obtain
	\begin{equation}
	\label{eq:pathfirstmoment4}
	\begin{split}
	\Pd \left( s_1\leftarrow s_2\leftarrow\dots\leftarrow s_k   \right)
	&\leq 
	f(s_1)\exp\left\{   \sum_{m=s_{1}}^{s_k-2} \left(\frac{1}{m}-\frac{f(m+1)}{2m}\right)   \right\}\prod_{m=2}^k\frac{f(s_m)}{2(s_{m}-1)}
	\\
	&\leq 
	f(s_1)\frac{s_k-1}{s_1+1}\prod_{m=2}^k\frac{f(s_m)}{2(s_{m}-1)},
	\end{split}
	\end{equation}
	finishing the proof of the lemma.
\end{proof}
We are finally able to prove Lemma \ref{l:pathfirstmoment2}.
\begin{proof}[Proof of Lemma \ref{l:pathfirstmoment2}]
	We will use Lemma~\ref{l:pathfirstmoment} and then an application of the union bound. First, fix~$s_1,s_k\in\N$ such that~$t_0\leq s_1<s_k \leq t$. We have, by Stirling's approximation formula and the positivity of the terms involved,
	\begin{equation}
	\label{eq:pathfirstmoment6}
	\begin{split}
	\sum_{\substack{ s_2,s_3,\dots,s_{k-1} \\ s_1<s_2<\dots<s_{k-1}<s_k}}
	\prod_{m=2}^{k-1}\frac{f(s_m)}{s_m-1} & \le
	\frac{1}{(k-2)!}\left(    \sum_{j=t_0}^{t}\frac{f(j)}{j-1}    \right)^{k-2}
	\\
	& \le 
	C\exp\left\{    -(k-2)\left(\log(k-2)  -1  -\log\left(       \sum_{j=t_0}^{t}\frac{f(j)}{j-1} \right)        \right)                              \right\}.
	\end{split}
	\end{equation}
	We can then show, by the above equation, Lemma~\ref{l:pathfirstmoment}, and the union bound,
	\begin{equation}
	\label{eq:pathfirstmoment7}
	\begin{split}
	\lefteqn{\Ed\left[ |\mathcal{V}_{k,t_0}(t)| \right] }
	\phantom{*}
	\\
	&\le 
	\sum_{\substack{ s_1,\dots,s_{k} \\ t_0\leq s_1<\dots<s_k\leq t}}
	\Pd \left( s_1\leftarrow s_2\leftarrow\dots\leftarrow s_k   \right) 
	\\
	&\leq
	\sum_{\substack{ s_1,\dots,s_{k} \\ t_0\leq s_1<\dots<s_k\leq t}}
	f(s_1)\frac{s_k-1}{s_1+1}\prod_{m=2}^k\frac{f(s_m)}{2(s_{m}-1)}
	\\
	&=
	\frac{1}{2^{k-1}}\sum_{\substack{ s_1,s_k \\ t_0\leq s_1<s_k\leq t}}\frac{f(s_1)f(s_k)}{s_1+1}
	\sum_{\substack{ s_2,\dots,s_{k-1} \\ s_1<s_2<\dots<s_{k-1}<s_k}}\prod_{m=2}^{k-1}\frac{f(s_m)}{s_{m}-1}
	\\
	&\leq
	C\exp\left\{    -(k-2)\left(\log(k-2)  +C_2  -\log\left(       \sum_{j=t_0}^{t}\frac{f(j)}{j-1} \right)        \right)                              \right\} \sum_{\substack{ s_1,s_k \\ t_0\leq s_1<s_k\leq t}}\frac{f(s_1)f(s_k)}{s_1+1}
	\\
	&\leq
	C\exp\left\{  2\log t  -(k-2)\left(\log(k-2)  +C_2  -\log\left(       \sum_{j=t_0}^{t}\frac{f(j)}{j-1} \right)        \right)                              \right\},
	\end{split}
	\end{equation}
	thus concluding the proof of the result.
\end{proof}
\section{The family of regularly varying functions: sharp bounds for the diameter}\label{s:rv}

In this section we explore our results when more information on $f$ is provided. In particular, we assume that $f$ satisfies condition (\ref{def:conditionRV}) for $\gamma \in (0,1)$ and prove sharp bounds for the diameter. Recall that this condition holds whenever
\begin{equation*}
\exists \gamma, \text{ such that }\forall a>0,\; \lim_{t \to \infty} \frac{f(at)}{f(t)} = \frac{1}{a^{\gamma}}.
\end{equation*}
Also recall that when $f$ satisfies the above condition for $\gamma >0$ we say it is a regular varying function with index of regular variation $-\gamma$. The case $\gamma =0$ is said to be slowly varying.

To prove the results under the assumption of regular variation our arguments rely on the theorems from Karamata's theory. In particular, the \textit{Representation Theorem} (Theorem 1.4.1 of \cite{BGT89Book}) and \textit{Karamata's theorem} (Proposition 1.5.8 of \cite{BGT89Book}). The former states that if~$f$ is a regularly varying function with index $-\gamma$, then there exists a slowly varying function~$\ell$ such that $f$ is of the form
\begin{equation}\label{eq:repthm}
f(t) = \frac{\ell(t)}{t^{\gamma}}.
\end{equation}
Whereas, the latter states that if~$\ell$ is a slowly varying function, then, for any $a < 1$
\begin{equation}\label{eq:karamata}
\int_{1}^{t}\frac{\ell(x)}{x^a}\mathrm{d}x \sim \frac{\ell(t)t^{1-a}}{1-a},
\end{equation}
and for~$a>1$, we have
\begin{equation}\label{eq:karamata2}
\int_{t}^{\infty}\frac{\ell(x)}{x^a}\mathrm{d}x \sim \frac{\ell(t)}{(a-1)t^{a-1}}.
\end{equation}
Our proofs will also rely on another result from Karamata's Theory (Corollary A.3 of \cite{ARSEdge17}) which assures that for all $\varepsilon >0$,
\begin{equation}\label{eq:limfeps}
\frac{f(t)}{t^{\varepsilon}} \stackrel{t\to \infty }{\longrightarrow}0.
\end{equation}
The next step is to prove the constant order of the diameter of the graphs generated by regularly varying functions and how it depends on the index of regular variation on infinity. This result is stated on Theorem~\ref{t:rv} and we provide a proof for it below.
\begin{proof}[Proof of Theorem~\ref{t:rv}] We begin proving the lower bound.

\noindent \underline{Lower bound:} By~(\ref{eq:repthm}) we have that there exists a slowly varying function $\ell$ such that $f(t) = \ell(t)/t^{\gamma}$. Moreover, (\ref{eq:limfeps}) implies that
\begin{equation}
\frac{\log \ell (t)}{\log t} \to 0.
\end{equation}
Thus, we have that
\begin{equation}
-\frac{\log t}{\log f(t)} = \frac{\log t}{\gamma\log(t) + \log \ell (t) } =\frac{1-o(1)}{\gamma}.
\end{equation}
Applying Theorem~\ref{t:lowerbounddiam} gives us the desired lower bound.

\noindent \underline{Upper bound:} By the Representation Theorem and (\ref{eq:limfeps}), we have that $f$ also satisfies condition (\ref{def:conditionS}). Just notice that
\begin{equation}
\sum_{s=1}^{\infty} \frac{f(s)}{s} = \sum_{s=1}^{\infty} \frac{\ell(s)}{s^{1+\gamma}} < \infty.
\end{equation}
And by Karamata's Theorem (Equation~\eqref{eq:karamata2} in particular) we have that
\begin{equation}
 \sum_{s=t^{\frac{1}{13}}}^{t} \frac{\ell(s)}{s^{1+\gamma}} \sim \frac{\ell(t)}{t^{\gamma\frac{\gamma}{13}}} \implies  - \frac{\log t}{\log \left( \sum_{s=t^{\frac{1}{13}}}^{t} \frac{\ell(s)}{s^{1+\gamma}}\right)} \sim \frac{\log t}{\frac{\gamma}{13}\log t + \log \ell(t)} = \frac{13(1-o(1))}{\gamma}
\end{equation}
And finally, applying Theorem~\ref{t:upperbounddiam} we prove the result.
\end{proof}

\section{Final comments}\label{s:finalcomments}
We end this paper with a brief discussion on the affine version of our model and how dropping some regularity conditions on $f$ may produce a sequence of graphs $\{G_t\}_{t \in \N}$ that has a subsequence which is essentially of complete graphs and another one which is close to the BA-model.

\subsection*{Affine version} At the introduction, we have discussed the affine version of the PA-rule, which we recall below.
\[
\Pd \left( v_{t+1} \rightarrow u \middle | G_t \right) = \frac{\text{degree}(u) + \d}{\sum_{w \in G_t}(\text{degree} (w) + \d)}.
\]
In \cite{ARSEdge17}, the authors showed that the effect of the affine term $\delta$ vanishes in the long run when one is dealing with the empirical degree distribution. I.e., their results show that $\delta$ has no effect on the degree sequence of the graphs, which is not observed in the affine version of the BA-model, for which the exponent of the power-law distribution depends on $\delta$, see \cite{CF03}. However, we believe $\delta$ may have an increasing/decreasing effect on the diameter's order. One also may find interesting to consider $\delta = \delta(t)$ and investigate which influence takes over: is it the edge-step function or the affine term?

\subsection*{Dropping regularity conditions}  In this part we discuss an example which hints that dropping some assumptions on $f$ may produce a somewhat pathological sequence of graphs. For instance, if we drop the assumption of $f$ being non-increasing, we may obtain a sequence of graphs whose diameter sequence oscillates between $1$ and $\log t$. More generally, the sequence of graphs oscillates between graphs similar to the BA-random tree, $\{\mathrm{BA}_t\}_{t \in \N}$, and graphs close to complete graphs.

Let $(t_k)_{k \in \N}$ be the following sequence: $t_0 = 1$ and $t_{k+1} := \exp\{t_k\}$, for $k>1$. Now, let $h$ be the edge-step function defined as follows
\begin{equation}
h(t) = \begin{cases}
1 & \quad \text{ if }t \in [t_{2k}, t_{2k+1}], \\
0 & \quad \text{ if }t \in (t_{2k+1}, t_{2k+2}).
\end{cases}
\end{equation}
The idea behind such $h$ is that between times $[t_{2k}, t_{2k+1}]$ the process behaves essentially as the traditional BA-model, whereas at interval $(t_{2k+1}, t_{2k+2})$ the process ``messes things up" connecting almost all vertices by only adding new edges. Moreover, in both regimes the process has time enough to ``forget about what was built in the past".

Using Lemma~\ref{l:deglowerbound} and reasoning as in (\ref{ineq:uvconnected}), one may prove that
\[
\diam_{G_{t_{2k+1}^{13}}(h)}G_{t_{2k+1}}(h) \le 2, \text{ a.a.s.}
\]
However, the process does not add any new vertex in the interval $(t_{2k+1}, t_{2k+2})$. Therefore,~$\diam G_{t_{2k+2}}(h) \le 2$, a.a.s.

On the other hand, we may collapse all vertices of $G_{t_{2k+2}}(h)$ into a single super vertex with $t_{2k+2}$ loops losing just one unit in its diameter. This way $\{G_{s}(h)\}_{t_{2k+2}}^{t_{2k+3}}$ is distributed as the BA-random tree started from the super vertex. Thus (see Theorem~$7.1$ of \cite{HBook2}), 
$$\diam G_{t_{2k+3}}(h) \approx  \log t_{2k+3} = t_{2k+2}.$$

Roughly speaking, $\{\diam G_t(h)\}_{t\geq 1}$ has a subsequence that is bounded by $2$ and another one that grows logarithmic in time. The conclusion is that if we drop some monoticity assumption on $f$, we may obtain a sequence of graphs having at least two subsequences that are completely different as graphs.

{\bf Acknowledgements } 
We are thankful to Roberto Imbuzeiro Oliveira for the many inspiring and pleasant conversations. \textit{C.A.}  was partially supported by Funda\c{c}\~{a}o de Amparo \`{a} Pesquisa do Estado de S\~{a}o Paulo (FAPESP), grants 2013/24928-2 and 2015/18930-0, and  by the DFG grant
SA 3465/1-1.  \textit{R.R.} was partially supported by Conselho Nacional de Desenvolvimento Cient\'\i fico e Tecnol\'{o}gico (CNPq) and by the project Stochastic Models of Disordered and Complex Systems. The Stochastic Models of Disordered and Complex Systems is a Millennium Nucleus (NC120062) supported by the Millenium Scientific Initiative of the Ministry of Science and Technology  (Chile).  \textit{R.S.} has been partially supported by Conselho Nacional de Desenvolvimento Cient\'\i fico e Tecnol\'{o}gico (CNPq)  and by FAPEMIG (Programa Pesquisador Mineiro), grant PPM 00600/16. 

\bibliography{ref}
\bibliographystyle{plain}

\end{document}